\documentclass[12pt]{amsart}
\usepackage{amsmath, amssymb, amsfonts, amsthm, latexsym,enumerate}

\newtheorem{Theorem}{Theorem}[section]
\newtheorem{Lemma}[Theorem]{Lemma}

\newtheorem{Proposition}[Theorem]{Proposition}

\theoremstyle{definition}
\newtheorem{Remark}[Theorem]{Remark}
\newtheorem{Example}[Theorem]{Example}

\def\To{\longrightarrow}
\def\reg{\operatorname{reg}}
\def\sdeg{\operatorname{sdeg}}
\def\depth{\operatorname{depth}}

\def\bideg{\operatorname{bideg}}

\def\height{\operatorname{ht}}

\def\mm{{\frak m}}
\def\M{{\mathcal M}}
\def\ZZ{{\mathbb Z}}

\def\CC{{\mathbb C}}
\def\QQ{{\mathbb Q}}

\begin{document}
\title{Regularity functions of powers of graded ideals}

\author{ Le Tuan Hoa}
\address{Institute of Mathematics, Vietnam Academy of Science and Technology, 18 Hoang Quoc Viet, 10072 Hanoi, Viet Nam}
\email{lthoa@math.ac.vn}

\author{Hop Dang Nguyen}
\address{Institute of Mathematics, Vietnam Academy of Science and Technology, 18 Hoang Quoc Viet, 10072 Hanoi, Viet Nam}
\email{ndhop@math.ac.vn}

\author{Ngo Viet Trung}
\address{Institute of Mathematics, Vietnam Academy of Science and Technology, 18 Hoang Quoc Viet, 10072 Hanoi, Viet Nam}
\email{nvtrung@math.ac.vn}


\begin{abstract}
This paper studies the problem of which sequences of non-negative integers arise as the functions $\reg I^{n-1}/I^n$, $\reg R/I^n$, $\reg I^n$ for an ideal $I$ generated by  forms of degree $d$ in a standard graded algebra $R$. These functions are asymptotically linear with slope $d$. If $\dim R/I = 0$, we give a complete characterization of all numerical functions which arise as the functions $\reg I^{n-1}/I^n$, $\reg R/I^n$  and show that $\reg I^n$ can be any numerical function $f(n) \ge dn$ that weakly decreases until it becomes a linear function with slope $d$. The latter result gives a negative answer to a question of Eisenbud and Ulrich.
If $\dim R/I \ge 1$, we show that $\reg I^{n-1}/I^n$ can be any numerical asymptotically linear function $f(n) \ge dn-1$ with slope $d$ and $\reg R/I^n$ can be any numerical asymptotically linear function $f(n) \ge dn-1$ with slope $d$ that is weakly increasing. Inspired of a recent work of Ein, Ha and Lazarsfeld on non-singular complex projective schemes, we also prove that the function of the saturation degree of $I^n$ is asymptotically linear for an arbitrary graded ideal $I$ and study the behavior of this function. 
\end{abstract}

\maketitle


\section*{Introduction} \smallskip

Let $R$ be a standard graded algebra over a field $k$, that is, an algebra generated by
finitely many forms of degree one. 
Let $M$ be a finitely generated graded $R$-module. The
Castelnuovo-Mumford regularity of $M$ is defined by
$$\reg M := \max\{a(H_\mm^i(M)) + i|\ i \ge 0\},$$
where $H_\mm^i(M)$ is the $i$-th local cohomology module of $M$ with respect to the maximal graded ideal $\mm$ of $R$, and $a(H)$ denotes the maximal non-vanishing degree of a graded Artinian module $H$ with the convention $a(H) = -\infty$ if $H = 0$.
If $R$ is a polynomial ring over a field, 
$$\reg M = \max\{b_i(M)-i|\ i = 0,...,r\},$$
where $b_i(M)$ denotes the maximum degree of the generators of the module $F_i$ in a minimal graded free resolution 
$$0 \longrightarrow F_r \longrightarrow \cdots \longrightarrow F_0 \longrightarrow M \longrightarrow 0$$
of $M$. As shown by the seminal works of Eisenbud and Goto \cite{EG} and Bayer and Mumford \cite{BM}, $\reg M$ controls the complexity of the graded structure of $M$. 

In 1991 Betram, Ein, and Lazarsfeld \cite{BEL} proved that if $I$ is a homogeneous ideal whose zero-locus defines a nonsingular complex projective scheme, then $\reg I^n$ is bounded above  by a linear function whose slope is the maximal generating degree $d(I)$ of $I$. 
This result has initiated the study on the function $\reg I^n$.  Later on, Cutkosky, Herzog, Trung \cite{CHT} and Kodiyalam \cite{Ko} independently found out that if $I$ is an arbitrary homogeneous ideal of a polynomial ring $R$, then $\reg I^n$ is an asymptotically linear function, i.e. there are integers $d$, $e$ and $n_0$ such that 
$$\reg I^n = dn  + e \ \text{for\ } n \ge n_0.$$
This result has been extended to the case $I$ is a graded ideal in a standard graded algebra $R$ by Trung and Wang \cite{TW}. The slope $d$ is called the {\it asymptotic degree} of $I$. It is the
smallest number $d$ such that $I^n = I_{\le d}I^{n-1}$ for large $n$, where $I_{\le d}$ denotes
the ideal generated by the elements of $I$ having degree at most $d$. 
However, it is very difficult to determine the intercept $e$ and the smallest number $n_0$ \cite{BHT,Ch1,Ch2,EH,EU,Ha,Ho1,Ro}. 

This paper addresses the behavior of the function $\reg I^n$ for $n <  n_0$. This problem was studied first by Eisenbud and Harris in 2010 \cite{EH}. They considered the sequence $e_n := \reg I^n - dn$, $n \ge 1$, which equals $e$ for $n \ge n_0$. If  $\dim R/I = 0$, and $I$ is {\it equigenerated}, i.e. $I$ is generated in a single degree, they proved that $\{e_n\}$ is a weakly decreasing sequence of non-negative integers.  
In 2012 Eisenbud and Ulrich \cite{EU} showed that if $H_\mm^0(R) = 0$, then $e_n - e_{n-1} \le d$. As far as we know,  this is the only case where we can say something about the behavior of sequence $\{e_n\}$ for $n < n_0$. 
Even this simple case is not well understood. For instance, Eisenbud and Ulrich \cite[p. 1222]{EU} asked whether the sequence $\{e_n - e_{n+1}\}$ is always weakly decreasing. 
This question has remained open until now.

If $\dim R/I > 0$, the sequence $\{e_n\}$ needs not be weakly decreasing.  For instance, if $R$ is a polynomial ring, Sturmfels \cite{St} found examples with $e_1 = 0 < e_2$. Later, Conca \cite{Co} gave examples with $e_1 = \cdots  = e_n = 0 < e_{n+1}$ for an arbitrary $n$. If $I$ is not equigenerated, Berlekamp \cite{Be} showed that the sequence $\{e_n\}$ can be initially increasing then later decreasing.  The above facts indicate that it would be difficult to describe the behavior of the sequence $\{e_n\}$.

Following \cite{Be} we call  $\{e_n\}$ the {\em defect sequence} of the function $\reg I^n$.
In both papers \cite{EH,EU}, the defect sequence of the function $\reg I^n$ was studied by means of the function $\reg R/I^n$.
Our idea is to estimate the behavior of the function $\reg R/I^n$ from that of the function $\reg I^{n-1}/I^n$,
which is easier to study because $I^{n-1}/I^n$ is a component of the associated graded ring of $I$. 

{\it Throughout this paper we assume that $R$ is a standard graded algebra over a field $k$ and $I$ is a graded proper ideal of $R$ for which $I^n \neq 0$ for $n \ge 1$}. 

It is easy to see that $\reg R/I^n$ and $\reg I^{n-1}/I^n$ are asymptotically linear functions. More precisely,
$\reg R/I^n = \reg I^{n-1}/I^n = dn + e -1$ for $n \gg 0$, where $d$ and $e$ are the slope and intercept of the asymptotically linear function $\reg I^n$. 
Set $a_n = \reg R/I^n - dn + 1$ and $c_n = \reg I^{n-1}/I^n - dn + 1$ for all $n \ge 1$.
We call $\{a_n\}$ and $\{c_n\}$ the {\em defect sequence} of the functions $\reg R/I^n$ and $\reg I^{n-1}/I^n$, respectively. It can be shown that the numbers $a_n$ and $e_n$ are 
non-negative integers if $I$ is an equigenerated ideal. The same holds for $c_n$ if $\height I > 0$.

Our first main results completely characterize the defect sequence of the function $\reg I^{n-1}/I^n$. \medskip

\noindent {\bf Theorems \ref{characterization1} and  \ref{ubiquity1}}.
{\it A sequence of non-negative integers is the defect sequence of the function $\reg I^{n-1}/I^n$ of an equigenerated ideal $I$ with $\dim R/I =0$ resp.~$\dim R/I \ge 1$ if and only if it is weakly decreasing resp.~convergent.}
\medskip

Consequently, any numerical asymptotically linear function $f(n)$ with slope $d$ is the function $\reg I^{n-1}/I^n$ of an ideal $I$ generated by forms of degree $d$ if and only if $f(n) \ge dn-1$.

For the function $\reg R/I^n$ we also have a complete characterization of the defect sequence in the case $\dim R/I = 0$.
\medskip

\noindent {\bf Theorem \ref{characterization2}}.
{\it A sequence of non-negative integers $\{a_n\}$ is the defect sequence of the function $\reg R/I^n$ of 
an ideal $I$ generated by forms of degree $d$ with $\dim R/I = 0$ if and only if it is weakly decreasing and
$a_n - a_{n+1} \le d$ for all $n \ge 1$.}
\medskip

We do not have a complete characterization of the defect sequence of the function $\reg R/I^n$ in the case $\dim R/I \ge 1$. However, we have the following simple sufficient condition, which shows that this defect sequence can fluctuate arbitrarily.
\medskip

\noindent {\bf Theorem \ref{ubiquity2}}.
{\it The defect sequence of the function $\reg R/I^n$ of an ideal $I$ generated by forms of degree $d$ with $\dim R/I \ge 1$ can be any convergent sequence of non-negative integers $\{a_n\}$ with the property $a_n - a_{n+1} \le d$ for all $n \ge 1$.}
\medskip

Consequently, the function $\reg R/I^n$ of an ideal $I$ generated by forms of degree $d$ can be any numerical asymptotically linear function $f(n) \ge dn-1$ with slope $d$ that is weakly increasing.
For instance, $\reg R/I^n$ can be arbitrarily larger than $n \reg R/I$ at any $n$.
This differs strikingly to the case $R$ is a polynomial ring, where it is difficult to find examples with $\reg I^n > n\reg I$ for some $n$ \cite{Co,St}. Note that we always have $\reg R/I^n = \reg I^n - 1$ if $I$ is a polynomial ideal.

For the defect sequence of the function $\reg I^n$ we obtain the following sufficient condition.
\medskip

\noindent {\bf Theorem \ref{ubiquity3}}.
{\it The defect sequence of the function $\reg I^n$ of an ideal $I$ generated by forms of degree $d$ with $\dim R/I = 0$ can be any weakly decreasing sequence $\{e_n\}$ of non-negative integers with the property $e_n - e_{n+1} \ge d$ for $n < m$, where $m$ is the least integer such that $e_n = e_{n+1}$ for all $n \ge m$.}
\medskip

In other words, the function $\reg I^n$ of  an ideal $I$ generated by forms of degree $d$ with $\dim R/I = 0$ can be any numerical function $f(n) \ge dn$ that weakly decreases first and then becomes a linear function with slope $d$. This statement also holds for $\dim R/I \ge 1$ by considering polynomial rings over $R$.
 
By Theorem \ref{ubiquity3}, the sequence $\{e_n - e_{n+1}\}$ can fluctuate or even increase before converging to zero. So we obtain a negative answer to the aforementioned question of Eisenbud and Ulrich.
There are reasons to conjecture that a sequence of non-negative integers $\{e_n\}$ is the defect sequence of the function $\reg I^n$ of an equigenerated ideal $I$ with $\dim R/I =0$ resp.~$\dim R/I \ge 1$ if and only if it is weakly decreasing resp.~convergent. 

The main difficulty in the proofs of the above theorems lies in the constructions of ideals with given defect sequences. Our idea is to choose $R$ of the form $k[X,Y]/Q$ and $I = (Y)^d+Q/Q$, where $X,Y$ are two sets of variables and $Q$ is a homogeneous ideal. Since the ideals $I^n$ is well determined, we only need to manipulate $Q$ to get the desired functions $\reg I^{n-1}/I^n$, $\reg R/I^n$, $\reg I^n$. 
These constructions cannot be applied to the case $R$ is a polynomial ring. 
In this case, there may be other constraints on the defect sequences. For instances, if $\dim R/I = 0$, the functions $\reg R/I^n$ and $\reg I^n$ must be strictly increasing by a result of Berlekamp \cite[Corollary 2.2]{Be}, whereas it needs not be so for non-polynomial rings by Theorems \ref{characterization2} and \ref{ubiquity3}.  

Another invariant related to the regularity is the {\it saturation degree} $\sdeg I$, which is the least number $t$ such that the saturation $\widetilde I := \cup_{n \ge 0} I:\mm^n$ coincides with $I$ in degree $\ge t$. We have $\sdeg I \le \reg R/I+1$ and equality holds if $\dim R/I = 0$. Recently, Ein, H\`a and Lazarsfeld \cite{EHL} proved that if $I$ is a homogeneous ideal whose zero-locus defines a non-singular complex projective scheme, then $\sdeg I^n$ is bounded above by a linear function whose slope is the maximal generating degree $d(I)$ of $I$. 
Inspired of their result, we prove that $\sdeg I^n$ is always asymptotically a linear function whose slope is $\le d(I)$ for any graded ideal $I$ in any standard graded algebra $R$. \medskip

\noindent {\bf Theorem \ref{sdeg}}. 
{Let $I$ be an ideal generated by forms of degree $d_1,...,d_p$. Then $\sdeg I^n$ is asymptotically a linear function with slope in $\{0, d_1,...,d_p\}$.}
\medskip

The above theorem follows from a more general result on the {\em partial regularity}
$$\reg_t M := \max\{a(H_\mm^i(M))+i|\ i \le t\},$$
where $t$ is any given number $\le \dim M$. If $R$ is a polynomial ring in $m$ variables, this notion is strongly related to the extremal Betti numbers introduced by Bayer, Charalambous and Popescu \cite{BCP}:
$$l\text{--}\reg M := \max\{b_i(M)-i|\ i \ge l\},$$
where $l$ is any given number $\le m$. By \cite{Tr2}, $\reg_t M = (m-t)$--$\reg M$.
We show that $\reg_t I^n$ is asymptotically a linear function with slope in $\{0,d_1,...,d_p\}$. 
If $\sdeg I^n$ is not asymptotically a constant function, we have $\sdeg I^n = \reg_1 I^n$ for $n \gg 0$.

If $I$ is an ideal generated by forms of degree $d$, $\sdeg I^n = dn + e$ for some $e \ge 0$, $n \gg 0$. 
We call the numbers $b_n := \sdeg I^n - dn$ the {\em defect sequence} of the function $\sdeg I^n$. 
If $\dim R/I = 0$, $\sdeg I^n = \reg R/I^n+1$. Hence Theorem \ref{characterization1} gives a complete characterization of the defect sequence of the function $\sdeg I^n$. If $\dim R/I = 1$, we obtain the following 
simple sufficient condition which shows that this sequence may fluctuate arbitrarily.
\medskip

\noindent {\bf Theorem \ref{sdeg2}}.
{\it The defect sequence of the function $\sdeg I^n$ of an ideal $I$ generated by forms of degree $d$ with $\dim R/I \ge 1$ can be any convergent sequence of non-negative integers $\{b_n\}$ with the property $b_n - b_{n+1} \le d$ for all $n \ge 1$.}
\medskip

Consequently, the function $\sdeg I^n$ of  an ideal $I$ generated by forms of degree $d$ can be any numerical asymptotically linear function $f(n) \ge dn$ with slope $d$ that is weakly increasing. 

The next sections of this paper will deal with the functions $\reg I^{n-1}/I^n$, $\reg R/I^n$, $\reg I^n$ and $\sdeg I^n$ separately in this order.


\section{The function $\reg I^{n-1}/I^n$}

In order to compare the functions $\reg I^n$ and $\reg I^{n-1}/I^n$ we need the following facts on the behavior of the Castelnuovo-Mumford regularity in a short exact sequence.

\begin{Lemma} \label{comparison1} {\rm \cite[Lemma 3.1(iii)]{HT}}
Let $0 \to N \to M \to E \to 0$ be a short exact sequence of finitely generated graded $R$-modules. 
Then $\reg E \le \max\{\reg N-1,\reg M\}$. Equality holds if $\reg N \ne \reg M$.
\end{Lemma}

\begin{Proposition} \label{asymptotic1}
Let $I$ be an arbitrary graded ideal. 
Then $\reg I^{n-1}/I^n = dn + e-1$ for $n \gg 0$, where $d$ and $e$ are the slope and intercept of the function 
$\reg I^n$ for $n \gg 0$.
\end{Proposition}

\begin{proof}
Since $\reg I^n = dn + e$ for $n \gg 0$,  $\reg I^n = \reg I^{n-1} + d$ for $n \gg 0$.  
Consider the short exact sequence 
$0 \to I^n \to I^{n-1} \to I^{n-1}/I^n \to 0.$
By Lemma \ref{comparison1}, it implies 
$$\reg I^{n-1}/I^n = \reg I^n - 1 = dn+e-1.$$
\end{proof}

As a consequence, the function $\reg I^{n-1}/I^n$ is always an asymptotically linear function.
This function has a further constraint if $I$ is an equigenerated ideal.

\begin{Proposition} \label{non-negative1}
Let $I$ be an ideal generated by forms of degree $d$ with $\height I > 0$.
Then $\reg I^{n-1}/I^n \ge dn -1$ for all $n \ge 1$.
\end{Proposition}

\begin{proof}
We first prove the case $\dim R/I = 0$. Since $I$ is not nilpotent, $I^{n-1}/I^n \neq 0$. 
Set $r = \reg I^{n-1}/I^n$. 
Since $I^{n-1}/I^n$ is an Artinian module, $r$ is its largest non-vanishing degree.
Thus, $r+1$ is the smallest integer such that $\mm^{r+1} \cap I^{n-1} \subseteq  I^n$.
Since $I$ is an $\mm$-primary ideal,  $\mm^{r+1} \cap I^{n-1}$ contain a power of $I$.
Therefore, $\mm^{r+1} \cap I^{n-1} \neq 0$ because $I$ is not nilpotent. 
Let $f \neq 0$ be a graded element of degree $r+1$ in $\mm^{r+1} \cap I^{n-1}$. 
Since $I^n$ is generated by elements of degree $dn$, the condition $f \in I^n$ implies $r+1 \ge dn$.
Hence $r \ge nd-1$. 

Let $\dim R/I > 0$. 
Without loss of generality we may assume that the base field $k$ is infinite. 
Then we can find a parameter ideal $Q$ in $R$ for $R/I$ that is generated by linear forms.
Let $M =  I^{n-1}/I^n$. Then $\dim M/QM = 0$. Hence there exists $t \ge 1$ such that $\mm^{t+1}M \subseteq QM$.
From this it follows that $\mm^{t+1}M = \mm^tQM$ because $Q$ is generated by linear forms. 
This means $Q$ is an $M$-reduction of $\mm$ (we refer to \cite{TW} for unexplained notions in this proof). By \cite[Lemma 1.2]{TW}, we may assume that $Q$ is generated by an $M$-filter-regular sequence of linear forms. Then we can apply \cite[Proposition 1.1]{TW} to get
$$\reg M \ge a(QM:\mm/QM) = \reg M/QM = \reg I^{n-1}/(QI^{n-1},I^n).$$
Since there exists a surjective map $I^{n-1}/(QI^{n-1},I^n) \to I^{n-1}/(Q\cap I^{n-1},I^n)$,
\begin{align*}
\reg I^{n-1}/(QI^{n-1},I^n) & \ge \reg I^{n-1}/(Q\cap I^{n-1},I^n)\\
& = \reg\, (I^{n-1},Q)/(I^n,Q).
\end{align*}

Consider the ideal $(I,Q)/Q$ in the ring $R/Q$. We have $\dim R/(I,Q) = 0$.
If $\height (I,Q)/Q = 0$, then $(I,Q)$ is contained in a minimal prime of $Q$. This prime ideal must be $\mm$. Hence,  
$\mm = \sqrt{Q}$. Since $Q$ is generated by $\dim R/I$ elements, we get 
$\dim R = \dim R/I$, which is a contradiction to the assumption $\height I > 0$.
Now we can apply the case $\dim R/I = 0$ to the ideal $(I,Q)/Q$ and obtain $\reg\, (I^{n-1},Q)/(I^n,Q) \ge dn-1$. Therefore,
$$\reg I^{n-1}/I^n  \ge \reg I^{n-1}/(QI^{n-1},I^n) \ge dn-1.$$
\end{proof}

Note that if $\dim R/I = 0$, then $\height I > 0$ because $I$ is non-nilpotent. 
We do not know whether Proposition \ref{non-negative1} holds if $\height I = 0$.

Set $c_n = \reg I^{n-1}/I^n - dn+1$ for all $n \ge 1$. 
By Proposition \ref{asymptotic1} and Proposition \ref{non-negative1}, $\{c_n\}$ is a convergent sequence of non-negative integers if $I$ is an equigenerated ideal with $\height I > 0$. 
We call $\{c_n\}$ the defect sequence of the function $\reg I^{n-1}/I^n$.
If $\dim R/I = 0$, there is an additional constraint on this defect sequence.

\begin{Proposition} \label{decreasing}
Let  $I$ be an ideal generated by forms of degree $d$ with $\dim R/I = 0$. Then the defect sequence of the function $\reg I^{n-1}/I^n$ is weakly decreasing.
\end{Proposition}

\begin{proof}
Set $r = \reg I^{n-1}/I^n$. We have seen in the proof of Proposition \ref{non-negative1} that
 $r+1$ is the smallest integer such that $\mm^{r+1} \cap I^{n-1} \subseteq  I^n$.
Since $I^{n-1}$ is generated by elements of degree $d(n-1)$, $\mm^{r+1} \cap I^{n-1} = \mm^{r+1-d(n-1)}I^{n-1}$. Similarly, $\mm^{r+d+1} \cap I^n = \mm^{r+d+1-dn}I^n$. Multiplying both sides of the relation $\mm^{r+1-d(n-1)}I^{n-1} \subseteq I^n$ with $I$ we obtain
$\mm^{r+1-d(n-1)}I^n \subseteq  I^{n+1}.$
This implies $\reg I^n/I^{n+1} \le r + d = \reg I^{n-1}/I^n + d$. Therefore, 
$$e_{n+1} = \reg I^n/I^{n+1} - d(n+1) +1 \le  \reg I^{n-1}/I^n - dn + 1 = e_n.$$
\end{proof}

It turns out that this additional constraint is exactly the condition for a convergent sequence of non-negative integers to be the defect sequence of the function $\reg I^{n-1}/I^n$  in the case $\dim R/I = 0$. 
To prove that we only need to show that any weakly decreasing sequence of non-negative integers is the defect sequence of the function $\reg I^{n-1}/I^n$ for an ideal generated by forms of degree $d$ with $\dim R/I = 0$. This follows from the following construction, where we consider the function $\reg I^n/I^{n+1}$ for convenience.

\begin{Proposition} \label{construction2}
Let $\{c_n\}_{n \ge 0}$ be any weakly decreasing sequence of positive integers and $d \ge 1$. 
Let $m$ be the minimum integer such that $c_n = c_m$ for $n > m$. Let $S = k[x,y]$ and 
$$Q = (x^{c_0},x^{c_1}y^d, ...,x^{c_m}y^{dm}).$$ 
Let $R = S/Q$ and $I = (y^d,Q)/Q$. Then for all $n \ge 0$,
$$\reg I^n/I^{n+1} = d(n+1)+c_n-2.$$
\end{Proposition}

\begin{proof}
We have
$$
I^n/I^{n+1}  = (y^{dn},Q)/(y^{d(n+1)},Q) \cong S/(Q:y^{dn}, y^d)(-dn).
$$
If $n < m$, 
$(Q: y^{dn},y^d) = (x^{c_0},...,x^{c_n},y^d) = (x^{c_n},y^d)$ because $c_0 \ge \cdots \ge c_n$. Thus,
$$\reg I^n/I^{n+1} = \reg S/(x^{c_n},y^d) + dn = d(n+1) + c_n-2.$$
If $n \ge  m$, $(Q: y^{dn},y^d) = (x^{c_0},...,x^{c_m},y^d) = (x^{c_m},y^d).$ Hence
\begin{align*}
\reg I^n/I^{n+1} & = \reg S/(x^{c_m},y^d) + dn = d(n+1) + c_m-2\\
& = d(n+1) + c_n-2.
\end{align*}
\end{proof}

Combining the above propositions we immediately obtain the following characterization for the defect sequence of the function $\reg I^{n-1}/I^n$.

\begin{Theorem} \label{characterization1}
A sequence of non-negative integers is the defect sequence of the function $\reg I^{n-1}/I^n$ for an equigenerated ideal $I$ with $\dim R/I =0$ if and only it is a weakly decreasing sequence. 
\end{Theorem}

The idea for the construction of an ideal generated by forms of degree $d$ with a given defect sequence 
comes from the observation that we can pass the investigation on the function $\reg I^n/I^{n+1}$ to the case $d =1$. 

Let $G = \oplus_{n \ge 0}I^n/I^{n+1}$ and $P = \oplus_{n > 0}I^n/I^{n+1}$. Then 
$$I^n/I^{n+1} = (P^n/P^{n+1})(-n(d-1)).$$ 
Thus, the defect sequence of the function $\reg I^n/I^{n+1}$ can be read off from that of $\reg P^n/P^{n+1}$. 
We can be represent $G$ as a quotient ring $k[X,Y]/Q$, where $k[X,Y]$ is a polynomial in two sets of variables, and $P = (Y)+Q/Q$, which is generated by linear forms. Therefore, we can pass our investigation to this setting. 
Once we can construct such an ideal $P$ with a given function $\reg P^n/P^{n+1}$, we may set 
$I = P^d$. Then
$$\reg I^n/I^{n+1} = \max\{\reg P^t/P^{t+1}|\  nd \le t \le (n+1)d-1\}.$$
So we may manipulate the given function $\reg P^n/P^{n+1}$ to get an ideal generated in degree $d$ with a desired defect sequence. 

For the case $\dim R/I = 0$ we may choose $X$ and $Y$ to consist of only one variable $x$ and $y$. Then 
$$P^n/P^{n+1} \cong (y^n,Q)/(y^{n+1},Q) \cong k[x,y]/(y,Q:y^n).$$
This has led us to choose $Q$ as in Proposition \ref{construction2}.
\smallskip

If $\dim R/I \ge 1$, we will use a similar construction to show that there is no constraint other than the convergence on the defect sequence of the function $\reg I^{n-1}/I^n$.

\begin{Theorem} \label{ubiquity1}
A sequence of non-negative integers is the defect sequence of the function $\reg I^{n-1}/I^n$ of an equigenerated graded ideal $I$ with $\dim R/I \ge 1$ if and only it is a convergent sequence. 
\end{Theorem} 

This result can be reformulated as follows.

\begin{Theorem} 
A numerical function $f(n)$ is the function $\reg I^{n-1}/I^n$ of an ideal $I$ generated by forms of degree $d$ with $\dim R/I \ge 1$ if and only if $f(n) \ge dn-1$ for all $n \ge 1$ and $f(n)$ is asymptotically linear with slope $d$. 
\end{Theorem}

To prove Theorem \ref{ubiquity1}  we only need to show that any convergent sequence of positive integers is the defect sequence of the function $\reg I^n/I^{n+1}$ for an ideal generated by forms of degree $d$ with $\dim R/I = 1$. 

\begin{Proposition} \label{construction1}
Let $\{c_n\}_{n\ge 0}$ be any convergent sequence of positive integers and $d \ge 1$. Let $m$ be the minimum integer such that $c_n = c_m$ for all $n > m$. 
$S = k[x_1,x_2,y_1,...,y_m]$, $P = (y_1,...,y_m)$ and 
$$Q = \big(x_1^{c_0},x_1P^d,\sum_{i=1}^{m-1}(x_2^{c_i},P^d)y_i^{di},x_2^{c_m}y_m^{dm}\big).$$ 
Let $R = S/Q$ and $I = (P^d+Q)/Q$. Then for all $n \ge 0$,
$$\reg I^n/I^{n+1} = d(n+1)+c_n-2.$$ 
\end{Proposition}

We have $\dim R/I = 1$ in this construction. It can be easily extended to any higher dimension by adding new indeterminates to $R$. 

To prove the Proposition \ref{construction1} we need the following technical result.

\begin{Lemma} \label{monomial}
Let $S$ and $P$ be as above. Let $c$ be a positive integer. Let $E$ and $F$ be ideals generated by monomials of degree $dn$ in $y_1,...,y_m$ such that $EP^{d-1} \not\subseteq F$. Then 
$$\reg P^{dn}/(x_1P^{nd},x_2^cE,F,P^{d(n+1)}) = d(n+1) + c - 2.$$
\end{Lemma}

\begin{proof}
Since $x_2$ is a regular element for $P^{dn}/(x_1P^{nd},E,F,P^{d(n+1)})$, we have 
$$\reg P^{dn}/(x_1P^{nd},E,F,P^{d(n+1)}) = \reg\, P^{nd}/(x_1P^{nd},x_2P^{nd},E,F, P^{d(n+1)}).$$
Since $P^{dn}/(x_1P^{dn},x_2P^{dn},E,F, P^{d(n+1)})$ is an Artinian module, its regularity is the largest non-vanishing degree, which comes from a monomial in $P^{dn}$ not containing $x_1,x_2$. Therefore,
$$\reg P^{dn}/(x_1P^{dn},x_2P^{dn},E,F, P^{d(n+1)}) \le d(n+1)-1.$$

Let $M = (x_1P^{dn},E,F,P^{d(n+1)})/(x_1P^{dn},x_2^cE,F,P^{d(n+1)})$. Since $M$ is annihilated by $x_1,x_2^c$ and $P^d$, it is an Artinian module. Its largest non-vanishing degree comes from a monomial in $x_2^{c-1}EP^{d-1} \setminus F$ not containing $x_1$. Therefore,
$$\reg M = d(n+1) + c-2 \ge d(n+1)-1.$$
Note that $H_\mm^0(M) = M$ and $H_\mm^i(M) = 0$ for $i > 0$.
From the derived long exact sequence of local cohomology modules of the short exact sequence 
$$
0 \to M  \to P^{dn}/(x_1P^{dn},x_2^cE,F,P^{d(n+1)})
\to P^{nd}/(x_1P^{dn},E,F,P^{d(n+1)}) \to 0$$
we can deduce that 
\begin{align*}
\reg P^{dn}/(x_1P^{dn},x_2^cE,F,P^{d(n+1)}) & = \max\left\{\reg M, \reg P^{nd}/(x_1P^{dn},E,F,P^{d(n+1)}) \right\}\\  
& = d(n+1) + c-2.
\end{align*}
\end{proof}

\noindent{\em Proof of Proposition \ref{construction1}}.
For $n = 0$ we have 
$$\reg R/I = \reg S/(x_1^{c_0},P^d) = d+c_0-2.$$

For $n \ge 1$ we have
\begin{equation}
I^n/I^{n+1} = (P^{dn} + Q)/(P^{d(n+1)} + Q) \cong P^{dn}/(P^{dn} \cap Q+P^{d(n+1)}).
\end{equation}
Note that $(x_1^{c_0}) \cap P^{dn} = x_1^{c_0}P^{dn} \subseteq x_1P^{dn}$. Then
$$P^{dn} \cap  Q  = x_1P^{dn} + \sum_{i=1}^{m-1} P^{dn} \cap (x_2^{c_i},P^d)y_i^{di} + P^{dn} \cap (x_2^{c_m}y_m^{dm}).$$
If $i < n$, we have
\begin{equation}
P^{dn} \cap (x_2^{c_i},P^d)y_i^{di} = x_2^{c_i}y_i^{di}P^{d(n-i)} + y_i^{di}P^{d(n-i)} = y_i^{di}P^{d(n-i)}.
\end{equation}
If $i \ge n$, we have
\begin{equation}
P^{dn} \cap (x_2^{c_i},P^d)y_i^{di} = (x_2^{c_i},P^d)y_i^{di}.
\end{equation}

For $n < m$, $P^{dn} \cap (x_2^{c_m}y_m^{dm}) = (x_2^{c_m}y_m^{dm}).$ Therefore, using (2) and (3) we have
$$
 P^{dn} \cap  Q  = x_1P^{dn} + \sum_{i=1}^{n-1} y_i^{di}P^{d(n-i)} + \sum_{i=n}^{m-1} (x_2^{c_i},P^d)y_i^{di} + (x_2^{c_m}y_m^{dm}).
$$
Since $P^dy_n^{dn}+ \displaystyle \sum_{i=n+1}^{m-1} (x_2^{c_i},P^d)y_i^{di} + (x_2^{c_m}y_m^{dm}) \subseteq P^{d(n+1)},$ 
\begin{align*}
 P^{dn}/(P^{dn} \cap Q+P^{d(n+1)}) & = P^{dn}/(x_1P^{dn},x_2^{c_n}y_n^{dn}, \sum_{i=1}^{n-1} y_i^{di}P^{d(n-i)}, P^{d(n+1)}). 
\end{align*}
Note that $y_n^{dn}P^{d-1} \not\subseteq \displaystyle \sum_{i=1}^{n-1} y_i^{di}P^{d(n-i)}$  because $y_n^{d(n+1)-1} \not\in (y_1,...,y_{n-1})$. Then we can apply Lemma \ref{monomial} and obtain
$$\reg P^{dn}/(P^{dn} \cap Q+P^{d(n+1)}) = d(n+1) + c_n-2.$$
By (1) this implies $\reg I^n/I^{n+1} = d(n+1) + c_n-2$.

For $n \ge m$, $P^{dn} \cap (x_2^{c_m}y_m^{dm}) = x_2^{c_m}y_m^{dm}P^{d(n-m)}.$ Therefore, using (2) we have
$$ P^{dn} \cap  Q = x_1P^{dn} + \sum_{i=1}^{m-1} y_i^{di}P^{d(n-i)} + x_2^{c_m}y_m^{dm}P^{d(n-m)}.$$
From this it follows that
\begin{align*}
 P^{dn}/(P^{dn} \cap Q+P^{d(n+1)}) & = P^{dn}/\big(x_1P^{dn},x_2^{c_m}y_m^{dm}P^{d(n-m)}, \sum_{i=1}^{m-1} y_i^{di}P^{d(n-i)} ,P^{d(n+1)}\big). 
\end{align*}
Note that $y_m^{dm}P^{d(n-m+1)-1} \not\in \displaystyle \sum_{i=1}^{m-1} y_i^{di}P^{d(n-i)}$ because $y_m^{d(n+1)-1} \not\in (y_1,...,y_{m-1})$. Then we can apply Lemma \ref{monomial} and obtain 
$$\reg P^{dn}/(P^{dn} \cap Q+P^{d(n+1)}) = d(n+1) + c_m-2.$$
By (1) this implies $\reg I^n/I^{n+1} = d(n+1) + c_m-2.$
The proof of Proposition \ref{construction1} is now complete.
\qed


\section{The function $\reg R/I^n$}

We shall see that the function $\reg R/I^n$ have similar constraints as the function $\reg I^{n-1}/I^n$.

\begin{Proposition} \label{convergence2}
Let $I$ be an arbitrary graded ideal. 
Then $\reg R/I^n = dn + e-1$ for $n \gg 0$, where $d$ and $e$ are the slope and intercept of the function 
$\reg I^n$ for $n \gg 0$.
\end{Proposition}

\begin{proof}
Consider the short exact sequence
$$0 \to I^n \to R \to R/I^n \to 0.$$
Since $\reg I^n$ is a linear function for $n \gg 0$, $\reg I^n > \reg R$ for $n \gg 0$.
By Lemma \ref{comparison1}, this implies $\reg R/I^n = \reg I^n - 1$.
\end{proof}

\begin{Proposition} \label{non-negative2}
Let $I$ be an ideal generated by forms of degree $d$. 
Then $\reg R/I^n \ge dn-1$ for all $n \ge 1$.
\end{Proposition}

\begin{proof} 
We may represent $R$ as a factor ring $S/Q$, where $S$ is a polynomial ring and $Q$ is a homogeneous ideal of $R$. 
Let $P$ an ideal in $S$ such that $I = (P,Q)/Q$.
Then $R/I^n = S/(P^n,Q)$. Hence $\reg R/I^n = \reg S/(P^n,Q) = \reg\, (P^n,Q)-1$.
Since $I^n$ is generated in degree $dn$, $(P^n,Q)$ contains minimal homogeneous generators of degree $dn$. Therefore,
$\reg\, (P^n,Q) \ge dn$ by \cite[Proposition 4.1 and Corollary 3.3]{Tr1}. Hence $\reg R/I^n \ge dn-1$.
\end{proof}

Set $a_n = \reg R/I^n - dn+1$ for all $n \ge 1$. We call $\{a_n\}$ the defect sequence of the function $\reg R/I^n$. By Proposition \ref{non-negative2}, $\{a_n\}$ is a convergent sequence of non-negative integers if $I$ is an equigenerated ideal.

If $\dim R/I = 0$, we know by Eisenbud and Harris \cite[Proposition 1.1]{EH} that this defect sequence is weakly decreasing. Furthermore, it is easy to see that the function $\reg R/I^n$ is weakly increasing, which imposes a further constraint on this defect sequence.

\begin{Proposition} \label{decrease}
Let $\{a_n\}$ be the defect sequence of the function $\reg R/I^n$ of an ideal $I$ generated by forms of degree $d$ with $\dim R/I = 0$. Then $a_n -  a_{n+1} \le d$ for all $n \ge 1$. 
\end{Proposition}

\begin{proof}
Since $R/I^n$ is an Artinian module, $\reg R/I^n$ is its maximal non-vanishing degree.
Therefore, from the surjective map $R/I^{n+1} \to R/I^n$ we can deduce that $\reg R/I^n \le \reg R/I^{n+1}$.
Since $\reg R/I^n = dn + a_n -1$ for all $n \ge 1$, this implies $a_n \le a_{n+1} + d$.
\end{proof}

It turns out that the above constraints completely characterize the defect sequence of the function $\reg R/I^n$ in the case $\dim R/I = 0$.

\begin{Theorem} \label{characterization2}
A sequence of non-negative integers $\{a_n\}$ is the defect sequence of the function $\reg R/I^n$ of 
an ideal $I$ generated by forms of degree $d$ with $\dim R/I = 0$ if and only if it is weakly decreasing and $a_n - a_{n+1} \le d$ for all $n \ge 1$. 
\end{Theorem}

\begin{proof}
By Theorem \ref{characterization1}, 
there exists an ideal $I$ generated by forms of degree $d$ in a standard graded algebra $R$ with $\dim R/I = 0$ such that $\reg I^{n-1}/I^n = dn+a_n-1$ for all $n \ge 1$.
In particular, $\reg R/I = d+a_1-1$.

For $n \ge 2$ we may assume by induction that $\reg R/I^{n-1} = d(n-1)+a_{n-1} -1$. 
Since $a_{n-1} \le  a_n+ d$,
$$\reg R/I^{n-1} \le dn + a_n - 1 = \reg I^{n-1}/I^n.$$ 
Consider the short exact sequence 
$$0 \to I^{n-1}/I^n \to R/I^n \to R/I^{n-1} \to  0.$$
Since all modules have dimension 0, we have
$$\reg R/I^n = \max\{\reg I^{n-1}/I^n, \reg R/I^{n-1}\} = dn + a_n - 1.$$
\end{proof}

Theorem \ref{characterization2} does not hold if $R$ is a polynomial ring. In this case, the defect sequence of the function $\reg R/I^n$ must be strictly increasing if $\dim R/I = 0$ \cite[Corollary 2.2]{Be}. From this it follows that  $a_n - a_{n+1} < d$ for all $n \ge 1$. 
\smallskip

An alternative proof for Theorem \ref{characterization2} can be deduced from the following formula for $\reg R/I^n$ in the construction of Proposition \ref{construction2}.

\begin{Proposition} \label{alternative}
Let $\{c_n\}_{ n \ge 0}$ be any weakly decreasing sequence of positive integers. 
Let $m$ be the minimum integer such that $c_n = c_m$ for all $n > m$. Let $S = k[x,y]$ and 
$$Q = (x^{c_0},x^{c_1}y^d, ...,x^{c_m}y^{dm}).$$ 
Let $R = S/Q$ and $I = (y^d,Q)/Q$. Then for all $n \ge 1$,
$$\reg R/I^n =  \left\{ \begin{array}{l l}
\max\big\{d(i+1) + c_i - 2|\ i = 0,...,n-1\big\} & \text{ if }\ n \le m,\\
\max\big\{dn + c_m -2, d(i+1) + c_i - 2|\ i = 0,...,m-1\big\} & \text{ if }\ n > m.
\end{array} \right. $$
In particular, if $c_i - c_{i+1} \le d$ for $i < m$, then
$\reg R/I^n = dn+c_{n-1}-2$ for all $n \ge 1$.
\end{Proposition}

\begin{proof}
We have $R/I^n = S/(y^{dn},Q)$, which is an Artinian quotient ring by a monomial ideal. Thus, $\reg R/I^n$ is the maximal degree of the monomials not in $(y^{dn},Q)$.

For $n \le m$, $(y^{dn},Q) = (x^{c_0},x^{c_1}y^d, ...,x^{c_{n-1}}y^{d(n-1)},y^{dn}).$ 
For $i = 0,...,n-1$, we consider the monomials $x^ay^b \not\in (y^{dn},Q)$ with $di \le b < d(i+1)$.
It is clear that $x^{c_i-1}y^{d(i+1)-1}$ has maximal degree among them.
Therefore, the maximal degree of the monomials not in $(y^{dn},Q)$ is attained among the monomials $x^{c_i-1}y^{d(i+1)-1}$, $i = 0,...,n-1$. Hence
$$\reg R/I^n = \max\{d(i+1) + c_i -2|\ i = 0,...,n-1\}.$$

For $n > m$, $(y^{dn},Q) = (x^{c_0},x^{c_1}y^d, ...,x^{c_m}y^{dm},y^{dn}).$ The maximal degree of the monomials not in $(y^{dn},Q)$ is attained among the monomials $x^{c_m-1}y^{dn-1}, x^{c_i-1}y^{d(i+1)-1}$, $i = 0,...,m-1$. Hence
$$\reg R/I^n = \max\{dn + c_m -2, d(i+1) + c_i -2|\ i = 0,...,m-1\}.$$

If $c_i - c_{i+1} \le d$ for $i < m$, then 
$$\max\{d(i+1) + c_i -2|\ i = 0,...,n-1\} = dn + c_{n-1}-2$$
for $n \le m$. The above formulas yields $\reg R/I^n = dn+c_{n-1}-2$ for all $n \ge 1$.
\end{proof}

Theorem \ref{characterization2} does not hold in the case $\dim R/I \ge 1$. The following result shows that the defect sequence of $\reg R/I^n$ can fluctuate arbitrarily.

\begin{Theorem} \label{ubiquity2}
The defect sequence of the function $\reg R/I^n$ of an ideal $I$ generated by forms of degree $d$ with $\dim R/I \ge 1$ can be any convergent sequence of non-negative integers $\{a_n\}$ with the property $a_n - a_{n+1}\le d$ for all $n \ge 1$. 
\end{Theorem}
 
This result follows from the following construction that is different from Proposition \ref{construction1}.
For convenience we consider the function $\reg I^n/I^{n+1}$.

\begin{Proposition} \label{construction3}
Let $\{c_n\}_{n\ge 0}$ be any convergent sequence of positive integers and $d \ge 1$. Let $m$ be the minimum integer such that $c_n = c_m$ for all $n > m$. Let $S = k[x_1,x_2,y_1,...,y_m]$, $P = (y_1,...,y_m)$ and 
$$Q = \big(x_1^{c_0},x_1x_2,x_1P^d,\sum_{i=1}^{m-1}(x_2^{c_i},P^d)y_i^{di},x_2^{c_m}y_m^{dm}\big).$$ 
Let $R = S/Q$ and $I = (P^d+Q)/Q$. Then for all $n \ge 1$,
$$\reg R/I^n = \left\{ \begin{array}{l l}
\max \{d(i+1) + c_i - 2|\ i = 0,...,n-1\} & \text{ if }\ n \le m,\\
\max \{dn+c_m-2, d(i+1) + c_i - 2|\ i = 0,...,m-1\}  & \text{ if }\ n > m.
\end{array} \right.$$ 
\end{Proposition}

\begin{proof}
We have $R/I^n  = S/(Q,P^{dn})$. 
Let 
$$Q^* := \bigcup_{t \ge 0}(Q,P^{dn}):\mm^t.$$ 
Then $H_\mm^0(R/I^n) = Q^*/(Q,P^{dn}).$ Hence
$$\reg R/I^n = \max\{\reg Q^*/(Q,P^{dn}), \reg S/Q^*\}.$$
Since $x_2$ is a parameter system in $S/(Q,P^{dn})$, $Q^* = \bigcup_{t \ge 0}(Q,P^{dn}):x_2^t.$ 
We will use this formula to compute $Q^*$.

For $n \le m$, we have
$$Q^* = (x_1,y_1^d,...,y_{n-1}^{d(n-1)},P^{dn}).$$
Since $Q^*/(Q,P^{dn})$ is an Artinian module, its regularity is given by the maximal degree of the monomials of $Q^* \setminus (Q,P^{dn})$. Each monomial of $Q^* \setminus (Q,P^{dn})$ is divisible by $x_1$ or $y_i^{di}$ for
some $i = 0,...,n-1$.
It is clear that the monomial generators of $x_1^{c_0-1}P^{d-1}$ and $x_2^{c_i-1}y_i^{di}P^{d-1}$ has maximal degree among these monomials. From this it follows that
$$\reg Q^*/(Q,P^{dn}) = \max\{d(i+1)+c_i-2|\ i = 0,...,n-1\}.$$
Since $x_2$ is a regular element on $Q^*$,
$$\reg S/Q^* = \reg S/(Q^*,x_2).$$ 
Since $S/(Q^*,x_2)$ is an Artinian quotient ring by a monomial ideal, its regularity is given by the maximal degree of the monomials not in $(Q^*,x_2)$. This maximal degree is attained among the monomials of $P^{dn-1} \setminus P^{dn}$. Therefore,
$$\reg  S/(Q^*,x_2) = dn-1.$$

Note that $dn + c_{n-1} - 2 \ge dn-1$. Then
\begin{align*}
\reg R/I^n &  = \max\{\reg Q^*/(Q,P^{dn}), \reg S/(Q^*,x_2)\}\\
& = \max\{d(i+1)+c_i-2|\ i = 0,...,n-1\}.
\end{align*}

For $n > m$, we have
$$Q^* = (x_1,y_1^d,...,y_m^{dm},P^{dn}).$$
Similarly as above, the maximal degree of the monomials of $Q^* \setminus (Q,P^{dn})$ is attained among the monomial generators of $x_1^{c_0-1}P^{d-1}$, $x_2^{c_i-1}y_i^{di}P^{d-1}$, $i = 1,...,n-1$, and $x_2^{c_m-1}y_m^{dm}P^{d(n-m)-1}$. Therefore,
$$\reg Q^*/(Q,P^{dn}) = \max\{dn+c_m-2, d(i+1)+c_i-2|\ i = 0,...,m-1\}.$$
Similarly as above, we can deduce that
$$\reg  R/I^n = \max\{dn+c_n-2,d(i+1)+c_i-2|\ i = 0,...,m-1\}.$$
\end{proof}

\noindent {\em Proof of Theorem \ref{ubiquity2}}.
Let $c_n = a_{n+1}+1$ for all $n \ge 0$. Then $c_n \le c_{n+1}+d$.
Let $m$ be the least integer such that $c_n = c_m$ for all $n > m$. 
If $n \le m$, we have 
$$\max \{d(i+1) + c_i - 2|\ i = 0,...,n-1\} = dn + c_{n-1} - 2 = dn + a_n-1.$$
If $n > m$, we have
$$\max\{dn +c_m-2,d(i+1)+c_i-2|\ i = 0,...,m-1\} = dn + c_m - 2 = dn + a_n-1.$$
By  Proposition \ref{construction3}, there exists an ideal $I$ generated by forms of degree $d$ in a standard graded algebra $R$ with $\dim R/I = 1$ such that $\reg R/I^n = dn+a_n-1$ for all $n \ge 1$.
For the case $\dim R/I > 1$ we only need to add new indeterminates to $R$.
\qed

By the definition of the defect sequence $\{a_n\}$ of the function $\reg R/I^n$, the condition $a_n - a_{n+1} \le d$ is equivalent to the condition $\reg R/I^{n+1} \ge \reg R/I^n$. Therefore, we can reformulate Theorem \ref{ubiquity2} as follows.

\begin{Theorem}
The function $\reg R/I^n$ of an ideal $I$ generated by forms of degree $d$ with $\dim R/I \ge 1$ can be any numerical asymptotically linear function $f(n) \ge dn-1$ with slope $d$ that is weakly increasing.
\end{Theorem}

The converse of this theorem does not hold because there are examples for which the function $\reg R/I^n$ is not weakly increasing.

\begin{Example} 
Nguyen and Vu \cite[Remark 5.9]{NV} have constructed an equigenerated ideal $I$ in a polynomial ring $R$ in $m \ge 4$ variables such that $\reg I = m + 3$ and $\reg I^n = 6n$ for $n \ge 2$. Therefore, if $m+3 > 6n$, 
$$\reg R/I = \reg I - 1 > \reg I^n - 1 = \reg R/I^n.$$
\end{Example}

\begin{Remark}
In the construction of Proposition \ref{construction3} we have $\reg R/I^n = \reg I^{n-1}/I^n$ for $n \ge 1$ if $c_n - c_{n+1} \le d$ for $n \ge 0$. Since the function $\reg I^{n-1}/I^n$ can be any numerical asymptotically linear function $f(n) \ge dn-1$ with slope $d$ by Theorem \ref{ubiquity1}, we conjecture that the same holds for the function $\reg R/I^n$.
\end{Remark}


\section{The function $\reg I^n$}

Let $I$ be an ideal generated by forms of degree $d$. 
We already know that there exists a number $e \ge 0$ such that $\reg I^n = dn +e$ for $n \gg 0$.
By \cite[Proposition 4.1 and Corollary 3.3]{Tr1}), $\reg I^n \ge dn$ for all $n \ge 1$.
Set $e_n = \reg I^n - dn$. Then $\{e_n\}$ is a convergent sequence of non-negative numbers.
We call $\{e_n\}$ the defect sequence of the function $\reg I^n$.

If $\dim R/I = 0$, Eisenbud  and Harris \cite[Proposition 1.1]{EH} showed that $\{e_n\}$ is a weakly decreasing sequence.
This sequence was also studied by Eisenbud and Ulrich \cite{EU}. They asked whether the sequence $\{e_n - e_{n+1}\}$ is always weakly decreasing \cite[p. 1222]{EU}.
Using the construction in Proposition \ref{construction2} we are able to give a negative answer to this question.

\begin{Proposition} \label{construction4}
Let $\{c_n\}_{ n \ge 0}$ be any weakly decreasing sequence of positive integers and $d \ge 1$. 
Let $m$ be the minimum integer such that $c_n = c_m$ for all $n > m$. Let $S = k[x,y]$ and 
$$Q = (x^{c_0},x^{c_1}y^d, ...,x^{c_m}y^{dm}).$$ 
Let $R = S/Q$ and $I = (y^d,Q)/Q$. Then for all $n \ge 0$,
$$\reg I^n =  \left\{ \begin{array}{l l}
\max\big\{d(i+1) + c_i - 2|\ i = n,...,m-1\big\} & \text{ if }\ n < m,\\
dn + c_n - 1 & \text{ if }\ n \ge m.
\end{array} \right. $$
\end{Proposition}

\begin{proof}
Note that $y$ is a parameter system for $S/Q$. Then 
$$\bigcup_{t \ge 0}Q:\mm^t = \bigcup_{t \ge 0}Q:y^t = (x^{c_0},...,x^{c_m}) = (x^{c_m}).$$
Hence $H_\mm^0(R) = (x^{c_m})/Q.$ Therefore,
$$\reg(R) = \max\{\reg H_\mm^0(R), \reg R/H_\mm^0(R)\} =   \max\{\reg (x^{c_m})/Q, \reg S/(x^{c_m})\}.$$

Since $(x^{c_m})/Q$ is an Artinian module, its regularity is the maximal degree of the homogeneous elements of $(x^{c_m}) \setminus Q$. For $i = 0,...,m-1$, we consider the monomials $x^ay^b \not\in Q$ with $c_{i+1} \le a < c_i$.
It is clear that $x^{c_i-1}y^{d(i+1)-1}$ has maximal degree among them. Hence
$$\reg (x^{c_m})/Q = \max\{d(i+1) + c_i -2|\ i = 0,...,m-1\}.$$
Since $\reg S/(x^{c_m}) = c_m-1 \le c_0-1 \le d+c_0-2,$ we get
$$\reg(R) = \max\{d(i+1) + c_i -2|\ i = 0,...,m-1\}.$$

For $n \ge 1$ we observe that
$I^n = (y^{dn},Q)/Q \cong  (S/Q:y^{dn})(-dn).$

If $n < m$, we have
$$Q:y^{dn} = (x^{c_n},x^{c_{n+1}}y^d,...,x^{c_m}y^{d(m-n)}).$$
Since $c_i \ge c_m$ for all $i < m$, it is easy to see that
\begin{align*}
H_\mm^0(S/Q:y^{dn}) & = (x^{c_m})/(x^{c_n},x^{c_{n+1}}y^d,...,x^{c_m}y^{d(m-n)})\\
& \cong S/(x^{c_n-c_m},x^{c_{n+1}-c_m}y^d,...,y^{d(m-n)})(-c_m).
\end{align*}
The maximal degree of the homogeneous elements of $S \setminus (x^{c_n-c_m},x^{c_{n+1}-c_m}y^d,...,y^{d(m-n)})$ is attained among the monomials $x^{c_i - c_m-1}y^{d(i-n+1)-1}$, $i = n,...,m-1$. Hence,
$$\reg H_\mm^0(S/Q:y^{dn}) = \max\{d(i-n+1) + c_i -2|\ i = n,...,m-1\},$$
which implies $\reg H_\mm^0(S/Q:y^{dn}) \ge c_m-1 = \reg S/(x^{c_m}).$

From the exact sequence 
$0 \to  H_\mm^0(S/Q:y^{dn}) \to S/Q:y^{dn}  \to S/(x^{c_m}) \to 0$
we get
$$\reg S/Q:y^{dn}  = \max\{\reg H_\mm^0(S/Q:y^{dn}), \reg S/(x^{c_m})\} = \reg H_\mm^0(S/Q:y^{dn}). $$
Therefore,
$$\reg I^n = \reg S/Q:y^{dn} + dn = \max\{d(i+1)+ c_i -2|\ i = n,...,m-1\}.$$

If $n \ge m$, we have $Q:y^{dn} = (x^{c_m}).$ Hence $\reg S/Q:y^{dn} = c_m-1$. Therefore,
$$\reg I^n = \reg S/Q:y^{dn} + dn = dn +  c_m -1  = dn +  c_n-1.$$
\end{proof}

Proposition \ref{construction4} leads us to the following sufficient condition for the defect sequence of the function $\reg I^n$ in the case $\dim R/I = 0$.

\begin{Theorem}\label{ubiquity3}
The defect sequence of the function $\reg I^n$ of  an ideal $I$ generated by forms of degree $d$ with $\dim R/I = 0$ can be any weakly decreasing sequence $\{e_n\}$ of non-negative integers with the property $e_n - e_{n+1} \ge d$ for $n < m$, where $m$ is the least integer such that $e_n = e_m$ for all $n > m$.
\end{Theorem}

\begin{proof}
We may consider $\{e_n\}$ as a subsequence of a weakly decreasing sequence $\{e^*_n\}$ of non-negative integers such that $e_n = e^*_{dn}$ for $n \le m$, 
$e^*_n \ge e^*_{n+1} +1$ for $n < md$ and $e^*_n = e^*_{dm}$ for all $n > dm$.

Let $c_0 \ge e^*_1$ be an arbitrary integer and $c_n = e^*_n+1$ for $n \ge 1$.
By Proposition \ref{construction4}, there is a graded ideal $J$ generated by linear forms 	
in a standard graded algebra $R$ with $\dim R/J = 0$ such that 
\begin{align*}
\reg J^n & =  \max\{i + c_i -1|\ i = n,...,dm-1\}\\
& = \max\{i + e^*_i|\ i = n,...,dm-1\} \text{ if }\ n < dm,
\end{align*}
and
$\reg J^n = n + c_n-1 = n + e^*_n$
if $n \ge dm$.
By the assumption we have $i + e^*_i-1 \ge i + e^*_{i+1}$ for $i < dm$.
Therefore, $\reg J^n = n + e^*_n$ if $n < dm$.

Set $I = J^d$. Then $\reg I^n = \reg J^{nd} = nd + e^*_{nd} = nd + e_n$ for all $n \ge 1$.
\end{proof}

It is easy to construct a sequence $\{e_n\}$ such that the conditions of Theorem \ref{ubiquity3} are satisfied and the sequence $\{e_n-e_{n+1}\}$ is increasing before it converges to zero. This gives a clear negative answer to the question of Eisenbud and Ulrich.

\begin{Example} 
Let $e_n = e_m+d(m-n) + (m-n)(n+m-1)/2$ for $n < m$ in Theorem \ref{ubiquity3}. Then $e_n-e_{n+1} = d + n$ for $n < m$. Hence $\{e_n-e_{n+1}\}$ is an increasing sequence for $n < m$.  
\end{Example}

\begin{Remark} 
The condition $e_n - e_{n+1} \ge d$ in Theorem \ref{ubiquity3} is opposite to the property $a_n - a_{n+1} \le d$ of the defect sequence of the function $\reg R/I^n$ in Proposition \ref{decrease}. 
If $H_\mm^0(R) = 0$, the defect sequence $\{e_n\}$ of the function $\reg I^n$ also has the property  $e_n - e_{n+1} \le d$ for all $n \ge 1$  \cite[Proposition 1.4(1)]{EU}. 
If $R$ is a polynomial ring, we must have $e_n - e_{n+1} < d$ by \cite[Corollary 2.2]{Be}.
\end{Remark}

For a defect sequence $\{e_n\}$ of the function $\reg I^n$, the condition  $e_n - e_{n+1} \ge d$ means $\reg I^n \ge \reg I^{n+1}$. Therefore, we obtain from Theorem \ref{ubiquity3} the following result.

\begin{Theorem}
The function $\reg I^n$ of an ideal $I$ generated by forms of degree $d$ with any given $\dim R/I \ge 0$ can be any numerical function $f(n) \ge dn$ that weakly decreases till it becomes a linear function with slope $d$.
\end{Theorem}

\begin{proof}
The case $\dim R/I = 0$ follows from Theorem \ref{ubiquity3}. 
The case $\dim R/I \ge 1$ is obtained by adding new indeterminates to $R$.
\end{proof}

The above theorem does not hold for a polynomial ideal. In this case, the function $\reg I^n$ is always strictly increasing if $\dim R/I = 0$ \cite[Corollary 2.2]{Be}.

It is easy to find examples with $\dim R/I \ge 1$ for which the function $\reg I^n$ fluctuates before becoming a linear function.

\begin{Example} 
Let 
$$R = \QQ[x,y,u,v]/(x^7,x^4y^3,x^3y^4,y^7,xu^6,yu^6,u^6v)$$ 
and $I = (xyv,u^3)R$. Then
$(\reg I, \reg I^2, \reg I^3, \reg I^4,\reg I^5) = (11,14,15,14,15)$.
This example was found by using Macaulay2 \cite{GS}.
\end{Example}

In spite of Theorems \ref{characterization1}  and \ref{ubiquity1} we conjecture that a sequence of non-negative integers $\{e_n\}$ is the defect sequence of the function $\reg I^n$ of an equigenerated ideal $I$ with $\dim R/I =0$ resp.~$\dim R/I \ge 1$ if and only if it is weakly decreasing resp.~convergent. 

\begin{Remark}
We always have $\reg I^n \ge d(I^n)$, where $d(I^n)$ is the maximal degree of the generators of $I^n$. The function $d(I^n)$ is asymptotically linear with the same slope $d$ as the function $\reg I^n$ \cite{CHT}. 
If $R$ is a polynomial ring, the function $d(I^n)-nd$ may have any given number of local maxima \cite{Ho2}. This is another indication that the defect sequence of the function $\reg I^n$ may behave wildly.
\end{Remark}

We can compute the functions $\reg I^n$ from the functions $\reg I^{n-1}/I^n$ or $\reg R/I^n$ under some circumstances. For that we shall need the following fact.

\begin{Lemma} \label{comparison2} {\rm \cite[Lemma 3.1(ii)]{HT}}
Let $0 \to N \to M \to E \to 0$ be a short exact sequence of finitely generated graded $R$-modules. 
Then $\reg N \le \max\{\reg M, \reg E+1\}$. 
Equality holds if $\reg M \neq \reg E$.
\end{Lemma}  

\begin{Proposition} \label{coincidence}
Let $\{c_n\}$ be the defect sequence of the function $\reg I^{n-1}/I^n$ of an ideal $I$ generated by forms of degree $d$. Assume  that $\reg R/I > \reg R$ and $c_n - c_{n+1} \le d-2$ for all $n \ge 1$. Then $\{c_n\}$ is the defect sequence of the function $\reg I^n$.
\end{Proposition}

\begin{proof} 
Consider the exact sequence $0 \to I \to R \to R/I \to 0$. 
By Lemma \ref{comparison2}, the assumption $\reg R/I > \reg R$ implies $\reg I = \reg R/I +1 = d+c_1$.
For $n \ge 2$ we consider the exact sequence $0 \to I^n \to I^{n-1} \to I^{n-1}/I^n \to 0$.
By induction we may assume that 
$\reg I^{n-1} = d(n-1) + c_{n-1}.$
Then
$$\reg I^{n-1}  < dn + c_n -1 = \reg I^{n-1}/I^n.$$
By Lemma \ref{comparison2}, this implies
$$\reg I^n = \reg  I^{n-1}/I^n + 1 = dn + c_n.$$
\end{proof}

\begin{Remark} 
We have seen in Theorem \ref{characterization1} that any weakly decreasing sequence $\{c_n\}$ of non-negative integers can be the defect sequence of a function $\reg I^{n-1}/I^n$. Hence, we may always choose $\{c_n\}$ such that $c_n - c_{n+1} \le d-2$ for all $n \ge 1$. However, the condition $\reg R/I > \reg R$ is not satisfied in the construction of Theorem \ref{characterization1}. By Proposition \ref{construction2} and Proposition \ref{construction4}, we have 
$$\reg R/I = d+c_0-2 \le \max\{d(i+1)+c_i-2|\ i = 0,...,m-1\} = \reg R.$$
Hence, we can not apply Proposition \ref{coincidence} in this case. 
For the same reason we can not apply Proposition \ref{coincidence} to the construction in Theorem \ref{ubiquity1}. 
\end{Remark} 

To compute  the function $\reg I^n$ from the function $\reg R/I^n$ we may use the exact sequence 
$$0 \to I^n \to R \to R/I^n \to 0.$$
By Lemma \ref{comparison2}, $\reg I^n = \reg R/I^n + 1$ if $\reg R < \reg R/I^n$ and $\reg I^n = \reg R$ if $\reg R > \reg R/I^n$. However, {\em if $\reg R = \reg R/I^n$, $\reg I^n$ can be arbitrary,  independent of the value of $\reg R/I^n$}.

\begin{Example}
Let $R$ and $I$ be as in Proposition \ref{construction4} with $m \ge 2$. Choose $\{c_n\}$ such that $c_0 \ge di+c_i$ for $i < m$.
Then $\reg R = d+c_0-2 = \reg R/I^n$ for $n < m$ by Proposition \ref{alternative}. However, 
$$\reg I^n = \max\{d(i+1) + c_i - 2|\ i = n,...,m-1\},$$
which can be any weakly decreasing sequence of integers between $dm-1$ and $d+c_0-2$ for $n = 1,...,m-1$.
\end{Example}


\section{The function $\sdeg I^n$}

Let $\widetilde I = \bigcup_{t \ge 0}I: \mm^t$ be the saturation of $I$. 
Then $\widetilde I/I = H_\mm^0(R/I)$ is of finite length. 
We denote by $\sdeg I$ the saturation degree  of $I$, which is the least number $t$ such that $\widetilde I$ agrees with $I$ in degree $t$ and higher. For convenience we set $\sdeg I = -\infty$ if $\widetilde I = I$.
We have
$$\sdeg I = a(\widetilde I/I) + 1 = a(H_\mm^0(R/I)) + 1.$$
We will show that $\sdeg I^n$ is an asymptotically linear function.
This follows from a more general result on partial regularities.

Let $M$ be a finitely generated graded $R$-module. 
For $t \ge 0$, we define 
$$\reg_t M := \max\{a(H_\mm^i(M))+i|\ i \le t\}.$$
In particular, $\reg_t M = \reg M$ for $t \ge \dim M$. This invariant has been studied in \cite{Tr1,Tr2,Tr3}.
We will show that $\reg_t I^n$ is an asymptotically linear function.

Let $I = (f_1,...,f_p)$, where $f_1,...,f_p$ are graded elements of $R$ of degree $d_1,...,d_p$, respectively. 
Let $R[It] := \oplus_{n \ge 0}I^n t^n$ be the Rees ring of $I$.
Then we can represent $R[It]$ as a quotient ring of a bigraded polynomial ring $S := R[y_1, . . . ,y_v]$, where $\bideg f = (\deg f, 0)$ for $f \in R$ and $\bideg y_i = (d_i , 1)$, $i = 1, . . . , v$.  

Let $\M$ be a finitely generated bigraded $S$-module. For any number $n$ let 
$$\M_n := \oplus_{a \in \ZZ}\M_{(a,n)}.$$
Then $\M_n$ is a  finitely generated graded $R$-module. 
We will show that $\reg_t \M_n$ is asymptotically a linear function. 

We first consider the case $\M_n$ has finite length for all $n \gg 0$. In this case,
$\reg_t \M_n = \reg \M_n = a(\M_n)$.
It was showed by Trung and Wang \cite[Proposition 2.1]{TW} that $a(\M_n)$ is asymptotically a linear function. Using the same approach, we can improve their result by giving more information on the slope of this linear function. Note that $a(\M_n) = -\infty$ means $\M_n = 0$.

\begin{Lemma} \label{a}
Assume that $\M_n$ has finite length for $n \gg 0$. Then either $a(\M_n) = -\infty$ for $n \gg 0$ or $a(\M_n)$ is asymptotically a linear function with slope in $\{d_1,..., d_v\}$.
\end{Lemma}

\begin{proof} 
We may assume that $\M_n \neq 0$ or, equivalently, $a(\M_n) \neq -\infty$ for $n \gg 0$.
Consider the exact sequence
$$0 \to (0_\M:y_v)_n \to \M_n \overset{y_v} \To \M_{n+1} \to (\M/y_v\M)_{n+1} \to 0.$$
Note that $0_\M:y_v$ and $\M/y_v\M$ are bigraded modules over $R[y_1,\ldots,y_{v-1}]$. 

If $v = 1$, $(0_\M:y_1)_n = 0$ and $(\M/y_1\M)_n = 0$ for $n \gg 0$. 
Hence $\M_n  \cong  \M_{n+1}(-d_1)$.
From this it follows that $a(\M_{n+1}) = a(\M_n) + d_1$ for $n \gg 0$. 
Therefore, $a(\M_n)$ is asymptotically a linear function with slope $d_1$. 

If $v > 1$, we may choose $v$ such that $d_v = \max\{d_i | i\ = 1, . . . , v\}.$ 
Using induction we may assume that $a((0_\M:y_v)_n)$ and $a((\M/y_v\M)_n)$ are asymptotically linear functions with slopes $\le d_v$. Then
\begin{align*}
a((0_\M:y_v)_n) + d_v & \ge a((0_\M:y_v)_{n+1}),\\
a(\M/y_v\M)_n) + d_v & \ge  a((\M/y_v\M)_{n+1}).
\end{align*}
for $n \gg 0$.

If $a(\M_n) = a((0_\M:y_v)_n)$ for $n \gg 0$, then $a((0_\M:y_v)_n) \neq -\infty$ for $n \gg 0$.
By induction, $a((0_\M:y_v)_n)$ is asymptotically a linear function with slope in $\{d_1,...,d_{v-1}\}$. 
Hence, $a(\M_n)$ is asymptotically a linear function with slope in $\{d_1,...,d_{v-1}\}$. 

It remains to consider the case that there exists an infinite sequence of
integers $m$ for which $a(\M_m) \neq a((0_\M:y_v)_m)$. 
By the above exact sequence, $a((0_\M:y_v)_n) \le a(\M_n)$ for all $n$. 
Hence $a((0_\M:y_v)_m) < a(\M_m)$. Therefore, we have an exact sequence
$$0  \to \M_{(a,m)} \overset{y_v} \To \M_{(a+d_v,m+1)} \to (\M/y_v\M)_{(a+d_v,m+1)} \to 0$$
for $a \ge a(\M_m)$.
From this it follows that 
$$a(\M_{m+1}) = \max\{a(\M_m) + d_v, a((\M/y_v\M)_{m+1})\}.$$ 

On the other hand, we have $$a(\M_n)+d_v \ge a(\M/y_v\M)_n + d_v \ge
a((\M/y_v\M)_{n+1})$$ for $n$ large enough. Since the above sequence of integers $m$ is infinite, there exists $m$ such that
$$a(\M_{m+1}) = a(\M_m) + d_v > a((0_\M:y_v)_m) + d_v \ge a((0_\M:y_v)_{m+1}).$$
Similarly as above, this implies
$$a(\M_{m+2}) = a(\M_{m+1}) + d_v >  a((0_\M:y_v)_{m+2})$$
and so on. So we get $a(\M_{n+1}) = a(\M_n) + d_v$ for $n \gg 0$.
Hence $a(\M_n)$ is asymptotically a linear function with slope $d_v$. 
\end{proof}

Note that $\reg_t \M_n = -\infty$ means $H_\mm^i(\M_n) = 0$ for $i = 0,...,t$.

\begin{Theorem} \label{partial1}
Let $\M$ be an arbitrary finitely generated bigraded $S$-module.
Then either $\reg_t \M_n = -\infty$ for $n \gg 0$ or $\reg_t \M_n$ is asymptotically a linear function with slope in
$\{d_1,..., d_v\}$. 
\end{Theorem}

\begin{proof} 
We only need to consider the case $t \le \dim R$. Without restriction assume that the base field $k$ is infinite. 
We will use a characterization of $\reg_t \M_n$ by means of a filter-regular sequence \cite[Proposition 2.2]{Tr1},.
This characterization was originally proved for $\reg_t R$ but the proof can be easily extended to the module case. According to this characterization, there exists a sequence of elements $z_1,...,z_{t+1} \in R$ such that
$(z_1,\ldots,z_i)\M_n:z_{i+1}/(z_1,\ldots,z_i)\M_n$ has finite length and
$$\reg_t \M_n =\max\big\{a\big((z_1,\ldots,z_i)\M_n:z_{i+1}/
(z_1,\ldots,z_i)\M_n\big)|\ i = 0,\ldots,t\big\}, $$
where $(z_1,\ldots,z_0) := 0$. The existence of such a sequence for all $n$ follows from the fact that we may assume that $\mm$ is generated by a filter-sequence for all $\M_n$, whose length is at least $\dim R$ \cite[Lemma 1.2]{TW}.

Let $\M^i = (z_1,\ldots,z_i)\M:z_{i+1}/(z_1,\ldots,z_i)\M$. 
By Lemma \ref{a}, either $a(\M^i_n) = -\infty$ or $a(\M^i_n)$ is asymptotically a linear function with slope $\delta_i \in \{d_1,...,d_n\}$. 
Since 
$$\reg_t \M_n = \max\{a(\M^i_n)|\ i = 0,..., t\},$$ 
either $\reg_t \M_n = -\infty$ or $\reg_t \M_n$ is asymptotically a linear function with slope in $\{d_1,...,d_v\}$.
\end{proof}

If $\M = R[It]$, then $\M_n \cong I^n$ for all $n$. Hence we immediately obtain from Theorem \ref{partial1} the following result on the function $\reg_t I^n$.

\begin{Theorem} \label{partial2}
Let $I$ be an ideal generated by forms of degree $d_1,...,d_p$. Then either $\reg_t I^n = -\infty$ for $n \gg 0$ or $\reg_t I^n$ is asymptotically a linear function with slope in $\{d_1,...,d_p\}$.
\end{Theorem}

For convenience, we will consider a function $f(n) = -\infty$ for $n \gg 0$ as an asymptotically linear function with slope 0. \smallskip

The saturation degree $\sdeg I^n$ is related to the partial regularity $\reg_1 I^n$ by the following observation.

\begin{Lemma} \label{relation}
~ \par
{\rm (1)} If $H_\mm^1(I^n) = 0$ for $n \gg 0$, then $\sdeg I^n = a(H_\mm^0(R))+1$ for $n \gg 0$. \par
{\rm (2)} If $H_\mm^1(I^n) \neq 0$ for $n \gg 0$, then $\sdeg I^n = \reg_1 I^n$ for $n \gg 0$.
\end{Lemma}

\begin{proof} 
From the short exact sequence $0 \to I^n \to R \to R/I^n \to 0$ we get the derived exact sequence
$$0 \to H_\mm^0(I^n) \to H_\mm^0(R) \to H_\mm^0(R/I^n) \to H_\mm^1(I^n) \to H_\mm^1(R).$$
Let $U$ be the saturation of the zero ideal of $R$. Then
$H_\mm^0(I^n) = U \cap I^n = 0$ for $n \gg 0$.

If $H_\mm^1(I^n) = 0$ for $n \gg 0$, then $H_\mm^0(R) \cong H_\mm^0(R/I^n)$. 
Hence, 
$$\sdeg I^n = a(H_\mm^0(R/I^n))+1 = a(H_\mm^0(R))+1$$ 
for $n \gg 0$.

If $H_\mm^1(I^n) \neq 0$ for $n \gg 0$, then $\reg_1 I^n \neq -\infty$ for $n \gg 0$.
By Theorem \ref{partial2}, $\reg_1 I^n$ is asymptotically a linear function. 
For $n \gg 0$, we have
$$\reg_1 I^n = \max\{a_0(H_\mm^0(I^n)),a(H_\mm^1(I^n))+1\} = a(H_\mm^1(I^n))+1.$$ 
From this it follows that $a(H_\mm^1(I^n))$ is asymptotically a linear function. Hence
$$a(H_\mm^1(I^n)) > \max\{a(H_\mm^0(R),a(H_\mm^1(R))\}$$
for $n \gg 0$. 
Therefore, the above derived exact sequence implies $a(H_\mm^0(R/I^n)) = a(H_\mm^1(I^n))$ for $n \gg 0$. So we obtain
$$\sdeg I^n = a(H_\mm^0(R/I^n)) +1= a(H_\mm^1(I^n))+1 = \reg_1 I^n$$
for $n \gg 0$.
\end{proof}

In Lemma \ref{relation}(1), we may consider $\sdeg I^n$ as an asymptotically linear function of slope 0 if $a(H_\mm^0(R)) = -\infty$.
Applying Lemma \ref{relation} we immediately obtain from Theorem \ref{partial2} the following result.

\begin{Theorem} \label{sdeg}
Let $I$ be an ideal generated by forms of degree $d_1,...,d_p$. Then $\sdeg I^n$ is asymptotically a linear function with slope in $\{0,d_1,...,d_p\}$.
\end{Theorem}



Ein, H\`a and Lazarsfeld \cite[Theorem A]{EHL} recently proved that if $R = \CC[x_0,...,x_r]$ is a polynomial ring over the complex numbers and $I = (f_0,...,f_p)$ an ideal generated by forms of degree $d_0 \ge \cdots \ge d_p$ such that the projective scheme cut out by $f_0,...,f_p$ is nonsingular, then 
$$\sdeg I^n \le d_0 n + d_1 + \cdots + d_r - r$$
for all $n \ge 1$. They asked whether this bound is best possible \cite[p.~1532]{EHL}. 
Theorem \ref{sdeg} suggests that this bound is not asymptotically optimal. In fact, the following example gives such an ideal $I$ for which $\sdeg I^n$ is a linear function for all $n \ge 1$ with slope and intercept arbitrarily less than those of this bound.

\begin{Example} \label{EHL}
Let $R = k[x_0,...,x_r]$, $Q = (x_0^2,...,x_r^2,x_0 \cdots x_r)$, and $I =  lQ$, where $l$ is an arbitrary linear form. Then the projective scheme cut out by the generators of $I$ is nonsingular. By the above result of Ein, H\`a and Lazarsfeld we have 
$$\sdeg I^n \le (r+2)n + 2r +3 \text{ for all } n \ge 1.$$ 
The function $\sdeg I^n$ can be computed explicitly. Since $\widetilde{I^n} = (l^n)$, we have
$$\widetilde I^n/I^n = l^n/l^nQ^n \cong (R/Q^n)(-n).$$ 
From this it follows that $\sdeg I^n = a(R/Q^n)+n+1.$
 Let $f$ be a monomial which has maximal degree among the monomials not contained in $Q^n$.
 We can always write $f = gh$, where $g$ is a product of the monomials $x_0^2,...,x_r^2$ and $h$ is a squarefree monomial.
 Since $fx_i \in Q^n$ for all $i = 0,...,r$, we must have $g \in (x_0^2,...,x_r^2)^{n-1} \subset Q^{n-1}.$ 
 Since $gh \not\in Q^n$, $\deg g = 2(n-1)$ and $h \not\in Q$. 
 Since $h$ has maximal degree among the squarefree monomials not contained in $Q$, $h$ is a product of $r$ variables. Hence $\deg f = 2(n-1) + r$. Therefore, $a(R/Q^n) = 2(n-1)+r$. So we get
 $$\sdeg I^n = 3n + r -1 \text{ for all } n \ge 1. $$
\end{Example}

In the remaining part of this paper, we address the problem which sequence of non-negative integers is the function $\sdeg I^n$ of an equigenerated ideal $I$. 


\begin{Proposition} \label{intercept}
Let $I$ be an ideal generated by forms of degree $d$. Assume that $\sdeg I^n$ is not a constant for $n\gg 0$.
Then there exists a number $b \ge 0$ such that $\sdeg I^n = dn + b$ for $n \gg 0$.
\end{Proposition}

\begin{proof}
By Theorem \ref{sdeg},  $\sdeg I^n$ is asymptotically a linear function with slope $d$.
It remains to show that $\sdeg I^n \ge dn$ for $n \gg 0$.

By Lemma \ref{relation}, $\sdeg I^n = \reg_1 I^n$ for $n \gg 0$.
As in the proof of Theorem \ref{partial1}, we may assume that there exists a linear form $x$ 
such that 
$$\reg_1 I^n = \max\big\{a((0:\mm)\cap I^n), a((xI^n : \mm)\cap I^n/xI^n)\big\}.$$
Since $H_\mm^1(I^n) \neq 0$ by Lemma \ref{relation}, $\reg_1 I^n \neq -\infty$. 
Hence $(0:\mm)\cap I^n \neq 0$ or 
$(xI^n : \mm)\cap I^n/xI^n \neq 0$. 
Since $I^n$ is generated by forms of degree $d n$, $a((0:\mm)\cap I^n) \ge dn$ or
$a((xI^n : \mm)\cap I^n/xI^n) \ge dn$. Therefore, $\reg_1 I^n \ge dn$. 
\end{proof}

With regard to Proposition \ref{intercept} we set $b_n := \sdeg I^n - dn$ for all $n \ge 1$. 
We call $\{b_n\}$ the defect sequence of the function $\sdeg I^n$. 



The number $b_n$ may be negative as shown by the following example.

\begin{Example} 
Let $S = k[x_0,...,x_{2t+1}]$, $Q = x_0(x_0,...,x_{2t+1})$ and
$$P = (x_1x_2,x_2x_3,...,x_{2t}x_{2t+1},x_{2t+1}x_1).$$
Let $R = S/Q$ and $I = (P,Q)/Q.$ Since $P$ is the edge ideal of a $(2t+1)$-cycle, we know by \cite[Example 3.4 and Theorem 3.5]{LT} that
$\depth k[x_1,...,x_{2t+1}]/P^n > 0$ for $n \le t$. From this it follows that 
$\widetilde{(P^n,Q)} = (x_0,P^n)$ for $n \le t$. Hence
$$\widetilde{I^n}/I^n =  (x_0,P^n)/(P^n,Q) \cong (x_0)/(x_0) \cap (P^n,Q) \cong (R/\mm)(-1).$$
Therefore, $\sdeg I^n = a(\widetilde{I^n}/I^n) + 1 = 2$ for $n \le t$. 
Since $d = 2$, we get $b_n = \sdeg I^n - 2n = 2-2n < 0$ for $n = 2,...,t$.
\par

For $n > t$, we observe that $x_1 \cdots x_{2t+1}(x_1x_2)^{n-t-1} \not\in P^n$ because all monomials of $P^n$ has degree $\ge 2n$. Therefore, $x_1 \cdots x_{2t+1}(x_1x_2)^{n-t-1} \not\in(P^n,Q)$ because all monomials of $Q$ contain $x_0$. It is easy to check that $x_1 \cdots x_{2t+1}(x_0,...,x_{2t+1}) \in (P,Q)$. Hence, $x_1 \cdots x_{2t+1}(x_1x_2)^{n-t-1} \in (P^n,Q)$. Therefore, $x_1 \cdots x_{2t+1}(x_1,x_{2t+1})^{n-t-1} \in \widetilde{(P^n,Q)}$.
From this it follows that $\widetilde{(P^n,Q)}/(P^n,Q)$ contains an element of degree $2n-1$. Therefore,  $\sdeg I^n \ge 2n$ for $n > t$. Hence, $b_n = \sdeg I^n - 2n \ge 0$ for $n > t$. 
\end{Example}

If $\dim R/I = 0$, $\sdeg I^n = \reg R/I^n+1$ for all $n \ge 1$. Hence, $\{b_n\}$ is the defect sequence of the function $\reg R/I^n$. By Proposition \ref{non-negative2}, $b_n \ge 0$ for all $n \ge 1$. 
By Theorem \ref{characterization2}, we get the following characterization for the defect sequence of the function $\sdeg I^n$ in the case $\dim R/I = 0$. 

\begin{Theorem}
A sequence of non-negative integers $\{b_n\}$ is the defect sequence of the function $\sdeg R/I^n$ of an ideal $I$ generated by forms of degree $d$ with $\dim R/I = 0$ if and only if it is weakly decreasing and $b_n - b_{n+1} \le d$ for all $n \ge 1$. 
\end{Theorem}

It remains to consider the case $\dim R/I \ge 1$. The following result shows that the defect sequence of the function $\sdeg I^n$ can fluctuate arbitrarily.

\begin{Theorem} \label{sdeg2}
The defect sequence of the function $\sdeg I^n$ of an ideal $I$ generated by forms of degree $d$ with $\dim R/I \ge 1$ can be any convergent sequence of non-negative integers $\{b_n\}$ with the property $b_n - b_{n+1}\le d$ for all $n \ge 1$. 
\end{Theorem}

\begin{proof} 
Let $c_n = b_{n+1}+1$ for $n \ge 0$. 
Consider the ideal $I$ in Proposition \ref{construction3} associated with the sequence $\{c_n\}_{n \ge 0}$.
By the proof of Proposition \ref{construction3},  
$\widetilde{I^n}/I^n = Q^*/(Q,P^{dn})$ and
$$\reg Q^*/(Q,P^{dn}) = \left\{ \begin{array}{l l}
\max\{d(i+1)+c_i-2|\ i = 0,...,n-1\}\! &\! \text{if}\ n \le m,\\
\max\{dn+c_m-2, d(i+1)+c_i-2|\ i = 0,...,m-1\}\! &\! \text{if}\ n > m,
\end{array} \right.$$
where $m$ is the minimum integer such that $c_n = c_m$ for all $n > m$.
Since  $c_{i-1}  \le c_i +d$, 
$$\begin{array}{l l}
\max\{d(i+1)+c_i-2|\ i = 0,...,n-1\} = dn + c_{n-1}-2 & \text{if } n \le m,\\
\max\{dn+c_m-2, d(i+1)+c_i-2|\ i = 0,...,m-1\} = dn+c_{n-1}-2 & \text{if } n > m.
\end{array}$$
Therefore, $\sdeg I^n = \reg Q^*/(Q,P^{dn})+1 = dn + c_{n-1} -1 = dn + b_n$ for all $n \ge 1$.
\end{proof}

Theorem \ref{sdeg2} can be also reformulated as follows.

\begin{Theorem}
The function $\sdeg I^n$ of an ideal $I$ generated by forms of degree $d$ with $\dim R/I \ge 1$ can be 
any numerical asymptotically linear function $f(n) \ge dn$ with slope $d$ that is weakly increasing. 
\end{Theorem}

The converse of the above theorem does not hold. The function $\sdeg I^n$ may fluctuate.

\begin{Example}
Let 
$$R = \QQ[x,y,u,v]/(x^6,x^3y^3,y^6,xu^5,yu^5,u^5v)$$ 
and $I = (xyv^4,u^6)R$. Then $(\sdeg I, \sdeg I^2, \sdeg I^4, \sdeg I^5) = (15,19,18,24,30)$.
This example was found by using Macaulay2 \cite{GS}.
\end{Example}

\medskip




\begin{thebibliography}{99}


\bibitem{BCP} D. Bayer, H. Charalambous and S. Popescu, {\it Extremal Betti numbers and applications to
monomial ideals},  J. Algebra {\bf 221} (1999), 497--512. 

\bibitem{BM} D. Bayer and D. Mumford, {\it What can be computed in algebraic geometry?}, Computational algebraic geometry and commutative algebra (Cortona, 1991), 1-48, Cambridge Univ. Press, 1993.

\bibitem{Be} D.   Berlekamp, {\it Regularity defect stabilization of powers of an ideal},  Math. Res. Lett. {\bf 19} (2012), 109 - 119.

\bibitem{BEL}  A. Bertram, L. Ein, and R. Lazarsfeld, {\it Vanishing theorems, a theorem of Severi, and the equations
defining projective varieties}, J. Amer. Math. Soc. {\bf 4} (1991), 587--602.

\bibitem{BHT} S. Bisui, Sankhaneel, H.T. H\`a  and A.C. Thomas, 
{\it Fiber invariants of projective morphisms and regularity of powers of ideals},
Acta Math. Vietnam. {\bf 45} (2020), 183--198.


\bibitem{Ch1} M. Chardin, 
{\it Powers of ideals and the cohomology of stalks and fibers of morphisms},
Algebra Number Theory {\bf 7} (2013), 1--18.

\bibitem{Ch2} M. Chardin, 
{\it Regularity stabilization for the powers of graded $\mm$-primary ideals},
Proc. Amer. Math. Soc. {\bf 143} (2015), 3343--3349.

\bibitem{Co} A. Conca, {\it Regularity jumps for powers of ideals}, Commutative Algebra with a focus on Geometric and Homological Aspects. Lecture Notes in Pure Applied Mathematics, {\bf 244}, 21--32. Chapman $\&$ Hall 2006.


\bibitem{CHT} D. Cutkosky, J. Herzog and N.V. Trung, {\it Asymptotic behavior of the Castelnuovo-Mumford regularity}, Compositio  Math. {\bf 118} (1999), 243 - 261.

\bibitem{EHL} L. Ein, H.T. H\`a and R. Lazarsfeld, 
{\it Saturation bounds for smooth varieties},
Algebra Number Theory {\bf 16} (2022), 1531--1546.

\bibitem{EG} D. Eisenbud and S. Goto, 
{\it Linear free resolutions and minimal multiplicity},
J. Algebra {\bf 88} (1984), 89--133.

\bibitem{EH} D. Eisenbud and J. Harris, {\it Powers of ideals and fibers of morphisms}, Math. Res. Lett. {\bf 17} (2010), 267 - 273.


\bibitem{EU} D. Eisenbud and B. Ulrich, {\it Notes on regularity  stabilization}, Proc.  Amer.  Math.  Soc. {\bf 140}  (2012),  1221--232.


\bibitem{GS} D. Grayson, M. Stillman,  Macaulay2, a software system for research in algebraic geometry,  available at http://www.math.uiuc.edu/Macaulay2/


\bibitem{Ha} H.T. H\`a, {\it Asymptotic linearity of regularity and $a^*$-invariant of powers of ideals},
Math. Res. Lett. {\bf 18} (2011), 1--9.
 
\bibitem{Ho1} L.T. Hoa, {\it Asymptotic behavior of integer programming and the stability of the Castelnuovo-Mumford regularity}, Math. Program. {\bf 193} (2022), Ser. A, 157--194.

\bibitem{Ho2} L.T. Hoa, {\it Maximal generating degrees of powers of homogeneous ideals}, Acta Math. Vietnam. {\bf 47} (2022), 19--37. 
 
 \bibitem{HT} L.T. Hoa and N.D. Tam, {\it On some invariants of  a mixed product of ideals}, Arch. Math. {\bf 94} (2010), 327--337.
 
 \bibitem{Ko} V. Kodiyalam, {\it    Asymptotic behaviour of Castelnuovo-Mumford regularity},  Proc. Amer. Math. Soc. {\bf  128} (2000),  407--411.
 
\bibitem{LT} H.M. Lam and N.V. Trung,
{\it Associated primes of powers of edge ideals and ear decompositions of graphs},
Trans. Amer. Math. Soc. {\bf 372} (2019), 3211--3236.
 
\bibitem{NV}  H.D. Nguyen,  and T. Vu, {\it Homological Invariants of Powers of Fiber Products}, Acta Mathematica Vietnamica {\bf 44} (2019), 617--638.
 
 \bibitem{Ro} T. R\"omer, {\it Homological properties of bigraded algebras}, Illinois J. Math. {\bf 45} (2001),
1361--1376. 
 
\bibitem{St} B. Sturmfels, {\it Four counterexamples in combinatorial algebraic geometry},  J. Algebra {\bf 230} (2000), 282--294.

\bibitem{Tr1} N.V. Trung, {\it Reduction exponent and degree bound for the defining equations of graded rings}, Proc. Amer. Math. Soc. {\bf 101} (1987), no. 2, 229–236.

\bibitem{Tr2} N.V. Trung, {\it Gr\"obner bases, local cohomology and reduction number}, Proc. Amer. Math. Soc. {\bf 129} (2001), no. 1, 9--18.

\bibitem{Tr3} N.V. Trung, {\it Evaluations of initial ideals and Castelnuovo-Mumford regularity}, Proc. Amer. Math. Soc. {\bf 130} (2002), no. 5, 1265--1274.

\bibitem{TW}  N.V. Trung and H-J. Wang, {\it On the asymptotic linearity of Castelnuovo-Mumford regularity},  J. Pure Appl. Alg. {\bf 201} (2005), 42--48.
 
\end{thebibliography}
\end{document}